\documentclass[11pt, reqno]{amsart}
\pdfoutput=1

\usepackage{amssymb}
\usepackage{mathtools}
\usepackage{mathrsfs}
\usepackage{stmaryrd}
\usepackage{upgreek}
\usepackage{centernot}
\usepackage{cancel}

\usepackage{bm}

\usepackage[T2A,T1]{fontenc}
\usepackage[utf8]{inputenc}
\usepackage[russian, french, english]{babel}
\usepackage{lmodern}
\usepackage[spacing=true,kerning=true,babel=true,nopatch=footnote]{microtype}
\usepackage[shortcuts]{extdash}
\usepackage[inline]{enumitem}
\usepackage{indentfirst}
\usepackage{xspace}
\usepackage{xcolor}
\usepackage{csquotes}

\usepackage[a4paper, left=2cm, right=2cm, top=1in, bottom=1in]{geometry}
\usepackage{graphicx}
\usepackage{wrapfig}
\usepackage{float}
\usepackage{array}
\usepackage{multicol}
\usepackage{mdframed}
\usepackage[skins]{tcolorbox}
\usepackage{framed}
\usepackage[justification=centering, labelfont=bf, font=small]{caption}
\usepackage{algorithm}

\usepackage{tikz}
\usetikzlibrary{shapes}
\usetikzlibrary{arrows.meta}
\usetikzlibrary{decorations.pathmorphing}
\usetikzlibrary{patterns}

\usepackage{titlesec}
\usepackage{titletoc}

\usepackage[
    sortcites,
    backend=biber,
    style=alphabetic,
    sorting=nyt,
    maxnames=100,
    backref=false
]{biblatex}
\usepackage[hidelinks]{hyperref}
\usepackage{cleveref}

\newcommand{\bemph}[1]{{\normalfont#1}}
\newcommand{\ep}[1]{\bemph{(}#1\bemph{)}}

\newtheoremstyle{bfnote}{}{}{\slshape}{}{\bfseries}{\bfseries.}{ }{\thmname{#1}\thmnumber{ #2}\thmnote{ \ep{\normalfont{}#3}}}
\theoremstyle{bfnote}
\newtheorem{theo}{Theorem}[section]
\newtheorem*{theo*}{Theorem}

\newtheorem{lemma}[theo]{Lemma}
\newtheorem{claim}[theo]{Claim}

\newtheorem{conj}[theo]{Conjecture}

\newtheorem*{corl*}{Corollary}

\theoremstyle{definition}

\newtheorem*{defn*}{Definition}

\newtheorem*{exmp*}{Example}

\theoremstyle{remark}
\newtheorem*{ques*}{Question}
\newtheorem*{remk*}{Remark}

\crefname{theo}{theorem}{theorems}
\Crefname{theo}{Theorem}{Theorems}
\crefname{lemma}{lemma}{lemmas}
\Crefname{lemma}{Lemma}{Lemmas}
\crefname{claim}{claim}{claims}
\Crefname{claim}{Claim}{Claims}
\crefname{prop}{proposition}{propositions}
\Crefname{prop}{Proposition}{Propositions}
\crefname{corl}{corollary}{corollaries}
\Crefname{corl}{Corollary}{Corollaries}
\crefname{defn}{definition}{definitions}
\Crefname{defn}{Definition}{Definitions}
\crefname{remk}{remark}{remarks}
\Crefname{remk}{Remark}{Remarks}
\crefname{alg}{algorithm}{algorithms}
\Crefname{alg}{Algorithm}{Algorithms}

\newcommand*{\myproofname}{Proof}
\newenvironment{claimproof}[1][\myproofname]{\begin{proof}[#1]}{\end{proof}}

\makeatletter
\newcommand{\neutralize}[1]{\expandafter\let\csname c@#1\endcsname\count@}
\newcommand{\authorfootnote}[1]{%
  \begingroup
  \renewcommand{\@makefnmark}{}%
  \renewcommand{\@makefntext}[1]{\noindent ##1}%
  \footnotetext{#1}%
  \endgroup
}
\makeatother

\newcommand{\set}[1]{\{#1\}}
\newcommand{\N}{{\mathbb{N}}}

\newcommand{\R}{\mathbb{R}}

\newcommand{\D}{\mathcal{D}}

\renewcommand{\P}{\mathbb{P}}
\newcommand{\E}{\mathbb{E}}
\renewcommand{\epsilon}{\varepsilon}
\newcommand{\eps}{\varepsilon}
\renewcommand{\phi}{\varphi}
\renewcommand{\leq}{\leqslant}
\renewcommand{\geq}{\geqslant}
\newcommand{\defeq}{\coloneqq}

\newcommand{\emphd}[1]{{\fontseries{b}\selectfont\textsf{#1}}}

\newcommand{\indicator}{\mathbf{1}}

\newcommand{\dr}[1]{d_{(#1)}}
\newcommand{\T}{\mathcal{T}}

\DeclareMathOperator{\Safe}{Safe}
\DeclareMathOperator{\Surv}{Surv}
\numberwithin{equation}{section}

\titleformat{\section}[block]{\large\bfseries\sffamily}{\thesection.}{1ex}{}
\titleformat{\subsection}[block]{\bfseries\sffamily}{\thesubsection.}{1ex}{}
\titleformat{\subsubsection}[block]{\itshape}{\bfseries\upshape\sffamily\thesubsubsection.}{1ex}{}

\titlespacing*{\section}{0pt}{*3}{*1}
\titlespacing*{\subsection}{0pt}{*3}{*1}
\titlespacing*{\subsubsection}{0pt}{*2}{*1}

\titlecontents{section}[1.5em]{\smallskip}{\bfseries\thecontentslabel\hspace{1.02em}}{\bfseries}{\,\,\titlerule*[0.77pc]{}\bfseries\contentspage}
\titlecontents{subsection}[4em]{\smallskip}{\thecontentslabel\hspace{1.02em}}{\hspace*{2.32em}}{\,\,\titlerule*[0.77pc]{.}\contentspage}

\renewbibmacro{in:}{}
\renewbibmacro*{volume+number+eid}{\printfield{volume}\setunit*{\addnbspace}\printfield{number}\setunit{\addcomma\space}\printfield{eid}}
\DeclareFieldFormat[article]{volume}{\textbf{#1}\space}
\DeclareFieldFormat[article]{number}{\mkbibparens{#1}}
\DeclareFieldFormat{journaltitle}{#1,}
\DeclareFieldFormat[thesis]{title}{\mkbibemph{#1}\addperiod}
\DeclareFieldFormat[article, unpublished, thesis, incollection, inproceedings]{title}{\mkbibemph{#1},}
\DeclareFieldFormat[book]{title}{\mkbibemph{#1}\addperiod}
\DeclareFieldFormat[unpublished]{howpublished}{#1, }
\DeclareFieldFormat{pages}{#1}
\DeclareFieldFormat[article]{series}{Ser.~#1\addcomma}

\setlist{topsep=3pt,itemsep=3pt}

\makeatletter
\newenvironment{breakablealgorithm}
  {\begin{center}\refstepcounter{algorithm}\hrule height .8pt depth 0pt \kern 2pt
   \renewcommand{\caption}[2][\relax]{%
     {\raggedright\textbf{\fname@algorithm~\thealgorithm} ##2\par}%
     \ifx\relax##1\relax
       \addcontentsline{loa}{algorithm}{\protect\numberline{\thealgorithm}##2}%
     \else
       \addcontentsline{loa}{algorithm}{\protect\numberline{\thealgorithm}##1}%
     \fi
     \kern 2pt\hrule\kern 2pt}}
  {\kern 2pt\hrule\relax\end{center}}
\makeatother

\title{\sffamily On Independent Sets in Uncrowded Uniform Hypergraphs}
\date{}

\author[J. Yu]{Jing~Yu}
\author[J. Zhang]{Junchi~Zhang}

\begin{document}

\begin{abstract}
We prove an average-degree lower bound on the independence number of uncrowded uniform hypergraphs.
For every fixed $r\geq 2$ and every $\eta>0$,  there exists $d_*=d_*(r,\eta)$ such that any uncrowded $(r+1)$-uniform hypergraph $G$ with $n$ vertices and average degree $d\geq d_*$ satisfies
\[
    \alpha(G)\geq
    (1-\eta)r^{-1/r}\left(\frac{\log d}{d}\right)^{1/r}n.
\]
The proof combines a cleaning procedure, which reduces the maximum top-layer degree to the average scale, with a random nibble procedure that repeatedly extracts independent vertices while controlling all lower-order degrees created by the process.
After an initial top-layer cleaning, we run a trace nibble.  Since the residual hypergraph contains traces of all sizes $2,\ldots,r+1$, we track the maximum degrees in every layer. A binomial-type recurrence for this degree profile
yields the stated leading constant.
\end{abstract}

\maketitle
\authorfootnote{(JY) Shanghai Center for Mathematical Sciences, Fudan University, Shanghai, China; 
e-mail: \texttt{jyu@fudan.edu.cn}. Partially supported by the National Natural Science Foundation of China grants 12371343 and 12525110 (PI: Hehui Wu).\\[2pt]
(JZ) Shanghai Center for Mathematical Sciences, Fudan University, Shanghai, China; 
e-mail: \texttt{jczhang24@m.fudan.edu.cn}.}

\section{Introduction}

\subsection{Background and the main result}
The \emphd{independence number} $\alpha(H)$ of a hypergraph $H$ is the size of its largest independent set. 
For $(r+1)$-uniform hypergraphs with average degree $d$, Caro and Tuza~\cite{CT91} gave a general lower bound of order $n/d^{1/r}$.
In sparse graphs and hypergraphs, local restrictions often force substantially larger independent sets. The first fundamental graph example is the triangle-free case: Ajtai, Koml\'os and Szemer\'edi~\cite{Ajtai1980Ramsy} proved a logarithmic improvement, and Shearer~\cite{Shearer1983ind1} obtained the sharp bound that every triangle-free graph $G$ with average degree $d$ satisfies
\[
    \alpha(G)\geq (1-o_d(1))\frac{n\log d}{d}.
\]

The corresponding problem for uniform hypergraphs was initiated by Ajtai, Koml\'os, Pintz, Spencer and Szemer\'edi~\cite{Spencer1982hyper}, who considered independent sets in \emph{uncrowded} hypergraphs.
For $k\geq 2$, a \emphd{Berge $k$-cycle}, denoted by $C_k$, is a sequence of distinct vertices $x_1,\ldots,x_k$ and distinct edges $e_1,\ldots,e_k$ such that $\{x_j,x_{j+1}\}\subseteq e_j$ for every $j$, with indices modulo $k$.
A hypergraph is called \emphd{uncrowded} if it contains no $C_k$ for $2\leq k \leq 4$.
Ajtai, Koml\'os, Pintz, Spencer and Szemer\'edi proved that, for every fixed $r$, there is a constant $c_r>0$ such that every $(r+1)$-uniform uncrowded hypergraph $H$ of maximum degree $\Delta$ satisfies
\[
   \alpha(H)\geq  c_r\left(\frac{\log \Delta}{\Delta}\right)^{1/r}|V(H)|.
\]
Duke, Lefmann and R\"odl~\cite{Rodl1995hyper} later developed a robust \emph{nibble method} in this setting. 

Verstraete and Wilson~\cite{Verstraete2026hyper} recently gave a Shearer-type proof in the broader class of \emphd{locally sparse} uniform hypergraphs, namely hypergraphs with no $C_2$ and no $C_3$.  Their theorem extends the preceding maximum-degree bound from uncrowded hypergraphs to locally sparse hypergraphs and gives an explicit constant; in particular, their proof gives the coefficient $(1-o_r(1))(1-o_\Delta(1))/r$ as $r,\Delta\to\infty$.  Their proof belongs to the same general Shearer-type line as the work of Kostochka, Mubayi and Verstraete~\cite{KostochkaMubayiVerstraete2014Independent}.

Our main theorem is the following.

\begin{theo}\label{thm:main}
For every integer $r\geq 2$ and every $\eta>0$, there exists $d_*=d_*(r,\eta)$ such that the following holds. 
Let $H$ be an $n$-vertex uncrowded $(r+1)$-uniform hypergraph with average degree $d\geq d_*$.  Then
\[
    \alpha(H)\geq
    (1-\eta)r^{-1/r}\left(\frac{\log d}{d}\right)^{1/r}n.
\]
\end{theo}

The leading constant in \Cref{thm:main} comes from the large-girth random-greedy benchmark.  For $d$-regular $(r+1)$-uniform hypergraphs of large girth, Nie and Verstraete~\cite{NV21} analyzed the randomized greedy algorithm and showed that, when the girth is sufficiently large relative to $d$, it produces an independent set of expected size
\[
    (1-o_d(1))r^{-1/r}\left(\frac{\log d}{d}\right)^{1/r}n .
\]
This value is also discussed in the concluding remarks of Verstraete and Wilson~\cite{Verstraete2026hyper}.  The graph case of this large-girth random-greedy picture goes back to Gamarnik and Goldberg~\cite{GamarnikGoldberg2010Greedy}; related asymptotic results for regular graphs of large girth were obtained by Lauer and Wormald~\cite{LauerWormald2007LargeGirth}.
The contribution of \Cref{thm:main} is to obtain this benchmark constant in an average-degree theorem under the uncrowded assumption.

\subsection{Basic definitions and notation}
Throughout the rest of the paper we use the following basic notation. For $n \in \N$, we let $[n] \defeq \set{1, \ldots, n}$.
An \emphd{$(r+1)$-bounded hypergraph} is a hypergraph whose edges have sizes between $2$ and $r+1$, and an \emphd{$(r+1)$-uniform hypergraph} is one whose edges all have size exactly $r+1$.  For an $(r+1)$-bounded hypergraph $H$ and $1\leq i\leq r$, let $E_i(H)$ be the set of edges of size $i+1$ and $E(H)=\bigcup_{i=1}^r E_i(H)$.
For $v\in V(H)$, define 
\[
    d_{i,H}(v)=|\{e\in E_i(H):v\in e\}|\quad \text{and} \quad
    D_i(H)=\max_{v\in V(H)}d_{i,H}(v),
\]
and
\[
   d_i(H)=\frac{1}{|V(H)|}\sum_{v\in V(H)}d_{i,H}(v)
       =\frac{(i+1)|E_i(H)|}{|V(H)|}.
\]
When the hypergraph is clear from the context, we write $d_i$ and $D_i$.
If $H$ is $(r+1)$-uniform, then $d_r(H)$ is its usual average degree.

For two disjoint vertex sets $A,K\subseteq V(H)$, the \emphd{trace subhypergraph} $H[K|A]$ of $H$ on $K$ with respect to $A$ is defined as 
$$V( H[K|A] )= K,\qquad E(H[K|A])= \{e\cap K:e\in E(H),\ e\subseteq A\cup K,\ |e\cap K|\geq 2\}.$$
Thus $H[K|A]$ records the traces left on $K$ by edges of $H$ whose remaining vertices lie in the activated set $A$.
The \emphd{induced subhypergraph} of $H$ on $K$, denoted by $H[K]$, is defined by
$$ V(H[K])=K,\qquad E(H[K])=\{e\in E(H)\;:\; e\subseteq K \} . $$

For convenience, we write $x=(1\pm\beta)y$ if $(1-\beta)y\leq x\leq (1+\beta)y$.

\subsection{Proof overview}
We now give an outline of the proof of \Cref{thm:main}.

\emph{Initial cleaning.} We begin with the original $(r+1)$-uniform hypergraph $H_0=H$.  If $D_r(H_i)>(1+\eps_0/2)d_r(H_i)$, we delete a vertex of maximum degree. 
The key point is that the quantity
\[
    \frac{|V(H_i)|}{d_r(H_i)^{1/r}}
\]
increases during such deletions.  Hence, if the average degree drops too much, or if too many vertices are deleted, a simple alteration lemma already gives the desired independent set.  Otherwise the cleaning phase stops with a uniform hypergraph $H_{\mathrm{cl}}$ satisfying
\[
    D_r(H_{\mathrm{cl}})\leq \left(1+\frac{\eps_0}{2}\right)d_r(H_{\mathrm{cl}}),
    \qquad
    \frac{|V(H_{\mathrm{cl}})|}{d_r(H_{\mathrm{cl}})^{1/r}}
    \geq
    \frac{n}{d^{1/r}}.
\]

\emph{Trace nibble.}
After the cleaning step stops, we begin with the final output $H_0'=H_{\mathrm{cl}}$.
In round $k$ of the nibble phase, suppose that the current residual hypergraph is $H'_k$ and write $d_{(k)}=d_r(H'_k)$.  We sample a $p_k$-random set $A$ with
\[
    p_k=h d_{(k)}^{-1/r}
\]
for some small enough parameter $h$.
From $A$ we keep a large independent set $I_k$.  A vertex outside $A$ is allowed to survive only if no incident edge has all its remaining vertices in $A$; additional equalizing coin flips make the survival probability essentially uniform.  If $K$ is the set of surviving vertices, the next temporary residual object is the trace hypergraph $H'_k[K|A]$.  Then any independent set of this new trace hypergraph can be joined to $I_k$ to form an independent set in the previous hypergraph.

The trace hypergraph is then pruned: for each $1\leq \ell\leq r$, vertices who are incident to too many edges of size less than $r+1$ are deleted.  Concentration shows that the total effect of these exceptional sets is negligible and the normalized layer-degree coefficients
\[
    c_{\ell,k}:=\frac{D_\ell(H'_k)}{d_{(k)}^{\ell/r}},\qquad 1\leq \ell\leq r,
\]
satisfy a polynomial recurrence, up to a small multiplicative error, as follows:
\[
    c_{\ell,k+1}
    \leq
    \sum_{m=\ell}^r c_{m,k}\binom{m}{\ell}h^{m-\ell}.
\]
This yields
\[
    c_{\ell,k}\lesssim \binom r\ell(kh)^{r-\ell}.
\]

\emph{Accumulating the independent set.}  We take $h=\rho(\log d^{1/r})^{-(r-1)/r}$ and run
\[
    T=\left\lfloor(1-10\eps_0)\frac{(\log d^{1/r})^{1/r}}{h}\right\rfloor
\]
nibble rounds.  The degree-profile estimate guarantees the errors and concentration in the nibble process, while the vertex and top-layer degree estimates keep
\[
    \frac{|V(H'_k)|}{d_{(k)}^{1/r}}\sim \frac{|V(H'_0)|}{d_{(0)}^{1/r}}.
\]
Summing the independent sets extracted in the $T$ rounds gives
\[
    |I|\geq
    (1-20r\rho-11\eps_0)
    r^{-1/r}\left(\frac{\log d}{d}\right)^{1/r}n.
\]
Choosing $20r\rho+11\eps_0<\eta$ completes the proof.

\subsection{Relation to prior work}

Our proof is based on a clean-then-nibble implementation of the R\"odl nibble.  In each nibble round we sample a small random set of vertices, keep a large independent subset of it, and pass to a residual trace hypergraph.  The main quantitative task is an inductive track of the maximum number $D_\ell(H_k')$ of edges of size $\ell+1$ incident to a vertex for $1\leq \ell\leq r$.

Since our theorem is stated in terms of the average degree, we start with a cleaning procedure inspired by the clean-nibble method of Dhawan, Janzer and 
Methuku~\cite{DhawanJanzerMethuku2025} to reduce the original hypergraph to be almost regular.
This is also closely related to the maximum-degree-to-average-degree transfer principle of the authors' earlier work~\cite{YuZhang2026Transfer}, which isolates the 
cleaning mechanism to convert a maximum-degree independence bound into an average-degree one. 

Then in the nibbling process, we keep the almost regularity for edges of size $r+1$. 
The main hypergraph difficulty is that a nibble step does not preserve
uniformity.  Edges whose remaining vertices survive the nibble leave
traces of all sizes $2,\ldots,r+1$.  We therefore control the whole
profile of layer-degrees by considering
\[
    c_{\ell,k}:=\frac{D_\ell(H'_k)}{d_{(k)}^{\ell/r}},
    \qquad 1\leq \ell\leq r.
\]
The one-step trace nibble gives a binomial-type recurrence for this profile, and this recurrence is what keeps the leading constant intact.

%Independently and concurrently, Dhawan, Methuku and Vo~\cite{Dhawan2026independence} obtained the same constant for uncrowded uniform hypergraphs.  

\subsection{Structure of the paper}

The rest of the paper is organized as follows. \Cref{sec:lemmas} collects the auxiliary tools.  After the concentration inequalities, we prove a basic alteration lemma and show that trace subhypergraphs of uncrowded hypergraphs remain uncrowded.  The main result of that section is the one-step trace nibble lemma, which includes the degree-truncation step described above.  The section ends with the binomial propagation lemma for the layer-degree coefficients.
\Cref{sec:mainproof} proves \Cref{thm:main}. We first perform the initial top-layer cleaning, then run the multi-round trace nibble, control the evolution of the layer-degree profile, and sum the independent sets produced during the nibble phase.

\medskip
\noindent\textbf{Note added.}
After this paper was written, we learned that Abhishek Dhawan, Abhishek Methuku and Minh-Quan Vo~\cite{Dhawan2026independence} had independently and concurrently obtained a proof of \Cref{thm:main}.
Their proof also keeps track of all lower-order degrees generated by the process. Their key technical innovation is a carefully chosen damping factor, which slows the accumulation of lower-order traces and gives refined control on the vertex survival rate, allowing the nibble to run for sufficiently many rounds.  Their average-degree statement is then obtained by applying the transfer principle in \cite{YuZhang2026Transfer}. By contrast, our proof is designed directly for the average-degree theorem and gives a self-contained average-degree proof.

\section{Auxiliary lemmas}\label{sec:lemmas}

We will use a few standard concentration inequalities. The first one is Talagrand's inequality:
\begin{theo}[Talagrand's inequality \cite{MR14}]\label{thm:talagrand}
Let $X$ be a non-negative integer-valued random variable, not identically zero, and suppose that $X$ is determined by independent trials $\T_1,\ldots,\T_n$.
Assume that, for some $\mu,\rho>0$, the following two conditions hold.
\begin{enumerate}[label=\textup{(T\arabic*)}]
    \item Changing the outcome of any one trial $\T_i$ can change $X$ by at most $\mu$.
    \item For every integer $s\geq 1$, if $X\geq s$, then there is a set of at most $\rho s$ trials certifying that $X\geq s$.
\end{enumerate}
Then, for every $t\geq 0$,
\[
\P\left[
    |X-\E X|
    \geq t+20\mu\sqrt{\rho\E X}+64\mu^2\rho
\right]
\leq
4\exp\left(
    -\frac{t^2}{8\mu^2\rho(\E X+t)}
\right).
\]
\end{theo}
We will also use the Chernoff bound: 
\begin{theo}[Chernoff bound]\label{thm:chernoff}
Let $X_1,\ldots,X_n$ be independent random variables taking values in
$\{0,1\}$, and put
\[
    X=\sum_{i=1}^n X_i,\qquad \mu=\E X.
\]
Then, for every $0<\delta\leq 1$,
\[
    \P(X\geq (1+\delta)\mu)
    \leq
    \exp\left(-\frac{\delta^2\mu}{3}\right),
\]
and
\[
    \P(X\leq (1-\delta)\mu)
    \leq
    \exp\left(-\frac{\delta^2\mu}{2}\right).
\]
\end{theo}

The next lemma is the basic deletion bound used several times.

\begin{lemma}\label{lem:randomind}
Let $H$ be an $n$-vertex $(r+1)$-bounded hypergraph with $d=d_{r}(H)\geq 1$.
Suppose that $d_{i}(H) \leq c_i d^{i/r}$ for every $1\leq i\leq r$, with $c_i>0$ and $c_r=1$.
Then
\[
    \alpha(H)\geq
    \frac{n}{2d^{1/r}}\min\left\{1,\min_{1\leq i\leq r}(2rc_i)^{-1/i}\right\}.
\]
\end{lemma}

\begin{proof}
Choose a random set $K\subseteq V(H)$ by keeping each vertex independently with probability
\[
    p=\min\left\{1,d^{-1/r}\min_{1\leq i\leq r}(2rc_i)^{-1/i}\right\}.
\]
Let $H'=H[K]$.
Deleting one vertex from each edge of $H'$ leaves an independent set, so
\[
    \alpha(H)\geq
    \E\bigl(|V(H')|-|E(H')|\bigr).
\]
For each $i$,
\[
    \E |E_i(H')|=p^{i+1}|E_i(H)|
    =\frac{d_i(H)p^{i+1}n}{i+1}.
\]
Hence
\begin{align*}
    \E\bigl(|V(H')|-|E(H')|\bigr)
    &= pn-\sum_{i=1}^r\frac{d_i(H)p^{i+1}n}{i+1} \\
    &\geq pn\left(1-\sum_{i=1}^r c_i(pd^{1/r})^i\right)\\
    &\geq pn\left(1-\sum_{i=1}^r c_i(2rc_i)^{-1}\right)\\
    &= \frac{pn}{2}.
\end{align*}
The choice of $p$ gives the desired bound.
\end{proof}

\begin{lemma}\label{lem:trace-uncrowded}
Let $H$ be uncrowded, let $A,K\subseteq V(H)$ be disjoint, and let $H_1=H[K|A]$. Then every induced subhypergraph of $H_1$ is uncrowded.
\end{lemma}

\begin{proof}
Passing to an induced subhypergraph only deletes vertices and edges, so it is enough to prove the assertion for $H_1$. For each trace edge $f\in E(H_1)$, fix an original edge $e\in E(H)$ with $f=e\cap K$.  Distinct trace edges are assigned distinct original edges, since one original edge has only one trace.

If two traces $f_1=e_1\cap K$ and $f_2=e_2\cap K$ meet in at least two vertices, then the original edges $e_1,e_2$ meet in at least two vertices. This is a $C_2$ in $H$, impossible.

Suppose that $f_1,f_2,f_3$ form a $C_3$ in $H_1$. Thus, with cyclic indices, $f_i\cap f_{i+1}=\{x_i\}$ and the vertices $x_1,x_2,x_3$ are distinct.
The original edges $e_1,e_2,e_3$ contain the same three intersections. If some pair $e_i,e_j$ has an additional common vertex, then that pair gives a $C_2$ in $H$. If no such additional intersection exists, then $e_1,e_2,e_3$ form a forbidden $C_3$ in $H$. Both alternatives are impossible.

Finally suppose that $f_1,f_2,f_3,f_4$ form a $C_4$ in $H_1$, with cyclic indices. Let $x_i\in f_i\cap f_{i+1}$ be the four distinct trace intersection vertices. The original edges $e_i$ contain the same vertices $x_i$. If a consecutive pair $e_i,e_{i+1}$ has an additional intersection besides $x_i$, then those two original edges form a  $C_2$. If a non-consecutive pair has an intersection $y$, say $y\in e_1\cap e_3$, then either $y$ coincides with one of the $x_i$, again giving two original edges with two common vertices, or the three edges $e_1,e_2,e_3$ form a $C_3$ through $x_1,x_2,y$ after deleting any case that already yields a $C_2$. The same argument applies to the other non-consecutive pair $e_2,e_4$. If no extra intersection occurs, then $e_1,e_2,e_3,e_4$ form a  $C_4$ in $H$. Thus a trace $C_4$ always lifts to a forbidden $C_4$ or to a shorter forbidden cycle in $H$, a contradiction.
\end{proof}

For the nibbling procedure in the proof of \Cref{thm:main}, we apply the following algorithm.

\begin{breakablealgorithm}
\caption{One-step nibble with degree truncation}\label{algorithm:one-step-nibble}
\begin{flushleft}
\textbf{Input}: An uncrowded $(r+1)$-bounded hypergraph $H$, the top-layer average degree $d=d_r(H)$ and the $i$-th-layer maximum degree $D_i=D_i(H)=c_i d^{i/r}$, $1\leq i\leq r$, and a parameter $\eps$.\\
\textbf{Output}: An independent set $I$ and an induced trace subhypergraph $H''$.
\end{flushleft}
\begin{enumerate}[itemsep=.16cm]
    \item Set
    \[
        p:=\eps d^{-1/r},\qquad \gamma:=\prod_{k=1}^r(1-p^k)^{D_k}.
    \]
    \item Choose a $p$-random set $A\subseteq V(H)$ and let $I$ be a largest independent set in $H[A]$.
    \item For every vertex $v$, independently expose an equalizing coin $\xi_v$ with
    \[
        \P(\xi_v=1)=\frac{\gamma}{\prod_{e\ni v}(1-p^{|e|-1})}.
    \]
    Let $\Safe(v)$ be the event that $v\notin A$ and no edge $e\ni v$ satisfies $e\setminus\{v\}\subseteq A$, and define
    \[
        K:=\{v:\Safe(v)\text{ holds and }\xi_v=1\}.
    \]
    \item Form the trace hypergraph $H'=H[K|A]$.
    \item Truncate the exceptional vertices in each layer as follows: For $1\leq i\leq r$, put
    \[
        \sigma_i:=\gamma^i d^{i/r}\sum_{k=i}^r c_k\binom{k}{i}\eps^{k-i}
    \]
    and delete all vertices
    \[
        S_i:=\left\{v\in K: d_{i,H'}(v)>\left(1+\frac1{20}d^{-\eps/10}\right)\sigma_i \right\}.
    \]
    Set $S:=\bigcup_{i=1}^rS_i$ and
    \[
        H'':=H'[K\setminus S].
    \]
    \item Return $(I,H'')$.
\end{enumerate}
\end{breakablealgorithm}
\vspace{0.2cm}

We prove that with high probability, the nibble procedure gives a large independent set and a subhypergraph with the required properties, and thus the following one-step nibble lemma holds.

\begin{lemma}\label{lem:nibble}
Fix an integer $r\geq 2$ and $\lambda\in(0,1/(10r))$.  There exists $\bar{d}=\bar{d}(r,\lambda)$ such that the following holds.  
Let $H$ be an $n$-vertex uncrowded $(r+1)$-bounded hypergraph and write $d=d_{r}(H)$ and $D_i=D_{i}(H)$.  Suppose that
\[
    D_i=c_i d^{i/r},\qquad 1\leq i\leq r.
\]
Let $\eps$ be a number satisfying
\[
    \lambda(\log d^{1/r})^{-(r-1)/r}\leq \eps\leq \frac{1}{10r},
\]
and set
\[
    p=\frac{\eps}{d^{1/r}},\qquad \gamma=\prod_{k=1}^r (1-p^k)^{D_k},\qquad \delta=\sum_{i=1}^r c_i\eps^i .
\]
If
\[
    d\geq \bar{d},
    \qquad n\geq d^7,
    \qquad \delta\leq \frac{1}{100r},\qquad c_r<1+\dfrac{1}{4r},
    \qquad c_i<d^{1/r},\quad (1 \leq i < r),
\]
then there is an independent set $I\subseteq V(H)$ and an uncrowded induced trace subhypergraph $H''$ with $V(H'')\subseteq V(H)\setminus I$ such that the following statements hold.
\begin{enumerate}
    \item For every independent set $I'$ of $H''$, the set $I\cup I'$ is independent in $H$;
    \item  $|I|\geq (1-3\eps-3\delta)pn$;
    \item  $|V(H'')|=(1\pm d^{-\eps/10})\gamma(1-p)n$ and
    \[
        d_{r}(H'')=\left(1\pm \frac{1}{20}d^{-\eps/10}\right)d\gamma^r;
    \]
    \item For every $1\leq i\leq r$,
    \[
        \frac{D_{i}(H'')}{{d_{r}(H'')}^{i/r}} \leq (1+d^{-\eps/10})\sum_{k=i}^r c_k\binom{k}{i}\eps^{k-i}.
    \]
\end{enumerate}
\end{lemma}

\begin{proof}
Let $\kappa_1,\ldots,\kappa_6>0$ denote constants depending only on $r$ and $\lambda$. We choose $\bar{d}=\bar{d}(r,\lambda)$ sufficiently large once and for all. Throughout the proof, every estimate of the form ``for $d$ sufficiently large'' is meant uniformly over all
\[ \eps\in\left[\lambda(\log d^{1/r})^{-(r-1)/r},\, \frac1{10r}\right].
\]
By increasing $\bar{d}=\bar{d}(r,\lambda)$, we may assume throughout that
\[
p=\eps d^{-1/r}\leq d^{-\eps/20}\leq \frac1{10^6r} \quad\text{and}\quad \eps\geq \lambda d^{-1/r}.
\]
Indeed, the first inequality follows from $\eps\leq 1/(10r)$, while the latter two follow from $\eps\geq \lambda(\log d^{1/r})^{-(r-1)/r}$.
We shall also use the following simple lower bound on $\gamma$.
Applying $1-x\geq e^{-\frac{x}{1-x}}$ for $x\in (0,1)$ 
with $x=p^k$, we get
\[
    1-p^k\geq \exp\left(-\frac{p^k}{1-p^k}\right).
\]
Therefore, using $D_k=c_kd^{k/r}$,
\[
\begin{aligned}
    \gamma &= \prod_{k=1}^r (1-p^k)^{D_k}  \\
    &\geq \exp\left(-\sum_{k=1}^r D_k\frac{p^k}{1-p^k}\right) \\
    &=\exp\left(-\sum_{k=1}^r c_k d^{k/r}\frac{p^k}{1-p^k}\right).
\end{aligned}
\]
Since $p=\eps d^{-1/r}$, we have
\[
    d^{k/r}p^k=\eps^k.
\]
Moreover, for every $k\in[r]$, $p^k\leq p$, and hence
\[
    1-p^k\geq 1-p.
\]
It follows that
\[
    \sum_{k=1}^r c_k d^{k/r}\frac{p^k}{1-p^k}
    =\sum_{k=1}^r c_k\frac{\eps^k}{1-p^k}
    \leq \frac1{1-p}\sum_{k=1}^r c_k\eps^k
    =\frac{\delta}{1-p}.
\]
Since we have chosen $\bar{d}$ so large that $p\leq 1/2$, we have
\[
    \frac1{1-p}\leq 2.
\]
Consequently,
\begin{equation}\label{eq:lowerboundforgamma}
    \gamma \geq \exp\left(-\frac{\delta}{1-p}\right) \geq e^{-2\delta}.
\end{equation}
Let
\[
\Delta:=\max_{v\in V(H)}\sum_{i=1}^r d_{i, H}(v)
\]
be the total maximum degree of $H$. Since $c_i<d^{1/r}$ for $i<r$, we have
\[
D_i=c_i d^{i/r}<d^{(i+1)/r}\leq d\qquad (i<r),
\]
while $D_r=c_r d<2d$. Hence
\[
\Delta\leq \sum_{i=1}^r D_i < (r-1)d+2d=(r+1)d.
\]

Choose a random set $A\subseteq V(H)$ by including each vertex independently with probability $p$.  Let $X_v$ be the indicator of the event $v\in A$, and let $I$ be a largest independent set in $H[A]$.

We first show that $I$ is large.  By the Chernoff bound,
\[
    \P\bigl(|A|<(1-\eps)pn\bigr)
    \leq \exp\left(-\frac{\eps^2pn}{2}\right)\leq \exp \left(  -\dfrac{\lambda^3n}{2d^4}  \right).
\]
Let $Y=|E(H[A])|$. Then $Y$ is determined by the independent trials $(X_v)_{v\in V(H)}$. Changing one trial changes $Y$ by at most $\Delta$. If $Y\geq s$, then it is enough to reveal all vertices of $s$ edges present in $H[A]$, so at most $(r+1)s$ trials certify $Y\geq s$. Moreover,
\[
    \E Y = \sum_{i=1}^r p^{i+1}|E_i(H)| = np\sum_{i=1}^r \frac{d_{i}(H)p^i}{i+1} \leq np\sum_{i=1}^r c_i\eps^i = \delta pn.
\]
Let
\[
B:=20\Delta\sqrt{(r+1)\E Y}+64\Delta^2(r+1).
\]
Since $\Delta\leq (r+1)d$, $\E Y\leq \delta pn\leq pn$, $p=\eps d^{-1/r}$, $n\geq d^7$, and $\eps\geq \lambda(\log d^{1/r})^{-(r-1)/r}$, we have
\[
    \frac{B}{\eps pn}=o_{r,\lambda}(1).
\]
Thus, by increasing $\bar d$, we may assume that
\[
    B<\eps pn.
\]

Now put $t=(2\delta+\eps)pn$. Since $\E Y\leq \delta pn$, the event
\[
Y>(2\eps+3\delta)pn
\]
implies
\[
Y>\E Y+t+B.
\]
Talagrand's inequality therefore gives
\[
\P\bigl(Y>(2\eps+3\delta)pn\bigr)\leq 4\exp\left(-\frac{(2\delta+\eps)^2p^2n^2}{8\Delta^2(r+1)(4\delta+2\eps)pn}\right).
\]
Using $\Delta\leq (r+1)d$, $p=\eps d^{-1/r}$, and
$\eps\geq \lambda d^{-1/r}$, the exponent is at least
\[
\frac{\lambda^2n}{16(r+1)^3d^5}.
\]
Hence
\[
\P\bigl(Y>(2\eps+3\delta)pn\bigr) \leq 4\exp\left(-\frac{\lambda^2n}{16(r+1)^3d^5}\right).
\]
By the union bound, the two estimates for $|A|$ and $Y$ fail with probability at most
\[
\exp\left(-\frac{\lambda^3n}{2d^4}\right)
+ 4\exp\left(-\frac{\lambda^2n}{16(r+1)^3d^5}\right)
\leq 5\exp\left(-\frac{\lambda^3n}{16(r+1)^3d^5}\right).
\]
Deleting one vertex from each edge of $H[A]$ leaves an independent set. 
Therefore, with probability at least
\[
    1-5\exp\left(-\frac{\lambda^3 n}{16(r+1)^3d^5}\right),
\]
we have
\[
    |I|=\alpha(H[A])
    \geq |A|-Y
    \geq (1-3\eps-3\delta)pn.
\]
Fix such a choice of $I$ whenever this event occurs.

Next we construct $H''$.  For every $v\in V(H)$, let $\xi_v$ be an independent Bernoulli random variable with
\[
    \P(\xi_v=1) = \frac{\gamma}{\prod_{e\ni v}(1-p^{|e|-1})}.
\]
This is well-defined because
\[
    \prod_{e\ni v}(1-p^{|e|-1}) \geq \prod_{k=1}^r (1-p^k)^{D_k} = \gamma.
\]
For $v\in V(H)$, let $\Safe(v)$ be the event that $v\notin A$ and there is no edge $e\in E(H)$ such that $v\in e$ and $e\setminus\{v\}\subseteq A$.  Define
\[
    K=\{v\in V(H): \Safe(v)\text{ holds and }\xi_v=1\}.
\]
Since $H$ has no $C_2$, the sets $e\setminus\{v\}$ are pairwise disjoint as $e$ ranges over the edges containing $v$.  Therefore, for every $v\in V(H)$,
\[
    \P(v\in K)=(1-p)\prod_{e\ni v}(1-p^{|e|-1})\frac{\gamma}{\prod_{e\ni v}(1-p^{|e|-1})}=(1-p)\gamma.
\]
Hence
\[
    \E |K|=(1-p)\gamma n.
\]

We shall use Talagrand's inequality to concentrate $|K|$.  The random variable $|K|$ is determined by the independent trials
\[
    \T_v=(X_v,\xi_v),\qquad v\in V(H).
\]
Changing one trial $\T_u$ can affect membership in $K$ only for $u$ and for vertices sharing an edge with $u$; hence it changes $|K|$ by at most $1+r\Delta$.  If $|K|\geq s$, then for each certified vertex $v\in K$ it is enough to reveal $(X_v,\xi_v)$, and one witness $u\in e\setminus\{v\}$ with $X_u=0$ for every edge $e\ni v$. This certifies both $v\notin A$, $\xi_v=1$, and the safety condition for $v$.  Thus $|K|\geq s$ has a certificate of size at most $(\Delta+1)s$.  Applying Talagrand's inequality gives
\[
    \P\left(\bigl||K|-(1-p)\gamma n\bigr|> \frac{1}{100}d^{-\eps/10}(1-p)\gamma n \right) \leq \exp\left(-\frac{\kappa_1\gamma n}{d^{4}}\right).
\]
Since $\gamma\geq e^{-2\delta}\geq e^{-r}$, with probability at least
$1-\exp(-\kappa_2n/d^4)$,
\[
    |K|=\left(1\pm \frac{1}{100}d^{-\eps/10}\right)(1-p)\gamma n.
\]

Set the trace subhypergraph $H'=H[K|A]$.
By \Cref{lem:trace-uncrowded}, every induced subhypergraph of $H'$ is uncrowded.
We shall eventually take $H''=H'[K\setminus S]$ for a small exceptional set $S$.

\begin{claim}\label{claim:subhypergraphindependent}
    For any induced subhypergraph $H_1$ of $H'$ and independent set $I_1$ in $H_1$, $I\cup I_1$ is independent in $H$.
\end{claim}

\begin{claimproof}
    Suppose that an edge $e\in E(H)$ satisfies $e\subseteq I\cup I_1$.  If $e\cap K=\varnothing$, then $e\subseteq I$, contradicting the choice of $I$.  If $|e\cap K|=1$, then the unique vertex in $e\cap K$ is not safe, a contradiction. If $|e\cap K|\geq 2$, then $e\cap K\subseteq I_1\subseteq V(H_1)$, and since $H_1$ is an induced subhypergraph of $H'$, the set $e\cap K$ is an edge of $H_1$, contradicting that $I_1$ is independent.
\end{claimproof}

We now compute the expected degree condition. Fix $1\leq i\leq k\leq r$, an edge $e\in E_k(H)$ and a vertex $v\in e$.
For $x\in e$, let $\Surv(x,e)$ be the event that $X_x=0$, $\xi_x=1$, and that for every edge $f\in E(H)$ with $x\in f$ and $f\neq e$, there is a vertex $y\in f\setminus\{x\}$ such that $X_y=0$.
Thus $\Surv(x,e)$ is the event that $x$ survives all dangers except possibly the edge $e$ itself. We remove the condition coming from $e$ so that the survival events for different vertices of the same edge are determined by disjoint sets of trials. Since $H$ contains no $C_2$,
\[
    \P(\Surv(x,e))
    =
    (1-p)
    \frac{\gamma}{\prod_{f\ni x}(1-p^{|f|-1})}
    \prod_{\substack{f\ni x\\ f\neq e}}(1-p^{|f|-1})
    =
    \frac{(1-p)\gamma}{1-p^k}.
\]

\begin{claim}\label{claim:Rx-disjoint}
If $x,x'\in e$ are distinct, then the collections of trials determining $\Surv(x,e)$ and $\Surv(x',e)$ are disjoint.  Consequently the events $\{\Surv(x,e)\}_{x\in e}$ are mutually independent.
\end{claim}

\begin{claimproof}
The event $\Surv(x,e)$ depends on $(X_x,\xi_x)$ and on the activation variables $X_w$ for vertices $w\in f\setminus\{x\}$, where $f\neq e$ is an edge containing $x$. Suppose, for a contradiction, that the determining sets for $\Surv(x,e)$ and $\Surv(x',e)$ meet. Then there are edges $f,g\neq e$ and a vertex $w$ such that $x,w\in f$ and $x',w\in g$; the cases where $w=x$ or $w=x'$ already give that $e$ shares two vertices with one of $f,g$, hence a  $C_2$. If $f=g$, then $f$ contains both $x$ and $x'$, so $e$ and $f$ meet in at least two vertices, again giving a  $C_2$. If $f\neq g$, then the three edges $e,f,g$ have consecutive intersections at $x,w,x'$. Any additional intersection among them gives a  $C_2$; otherwise they form a  $C_3$. This contradicts uncrowdedness. Thus the determining sets are disjoint. Since the underlying trials are independent and these determining sets are pairwise disjoint as $x$ varies over $e$, the claimed mutual independence follows.
\end{claimproof}

For any $U\in \binom{e\setminus\{v\}}{i}$, set $W=e\setminus(\{v\}\cup U)$ and let $M(e,v,U)$ be the event that $e$ is transformed into the edge
$\{v\}\cup U$ in $H'$.  Equivalently,
\[
    M(e,v,U)
    =
    \Surv(v,e)
    \cap
    \bigcap_{u\in U}\Surv(u,e)
    \cap
    \bigcap_{w\in W}\{X_w=1\}.
\]
By \Cref{claim:Rx-disjoint}, and because the variables $X_w$ with $w\in W$ do not occur in the trials determining $\Surv(x,e)$ for $x\in\{v\}\cup U$, these events are mutually independent.  It follows that
\[
    \P(M(e,v,U))
    =
    \left(\frac{(1-p)\gamma}{1-p^k}\right)^{i+1}
    p^{k-i}.
\]
Moreover, the event $M(e,v,U)$ implies $v\in K$.  Hence
\[
    \P(\{v\}\cup U\in E(H')\mid v\in K)
    =
    \frac{\P(M(e,v,U))}{\P(v\in K)}
    =
    \frac{(1-p)^i\gamma^i}{(1-p^k)^{i+1}}p^{k-i}.
\]
Since $H$ has no $C_2$, two distinct original edges cannot give the same trace edge containing $v$.  Thus the following computation counts trace edges without multiplicity. Consequently, for every $v\in V(H)$,
\[
    \E\bigl(d_{i,H'}(v)\mid v\in K\bigr)
    =
    \sum_{k=i}^r
    d_{k,H}(v)\binom{k}{i}
    \frac{(1-p)^i\gamma^i}{(1-p^k)^{i+1}}p^{k-i}.
\]
Set
\[
    \mu_i
    =
    \sum_{k=i}^r
    d_{k}(H)\binom{k}{i}
    \frac{(1-p)^i\gamma^i}{(1-p^k)^{i+1}}p^{k-i}.
\]
Since $1-p>1/2$ and $\gamma\geq e^{-2\delta}>1/2$, the term $k=r$ gives
\[
    \mu_i\geq \frac{\eps^r}{4^r}d^{i/r}\geq d^{\frac{1}{2r}}.
\]
Moreover,
\[
    \E\bigl((i+1)|E_i(H')|\bigr)=(1-p)\gamma n\mu_i\geq nd^{\frac{1}{3r}}.
\]

We next concentrate $T_i:=(i+1)|E_i(H')|$. Regard $T_i$ as a function of the independent trials $\T_u=(X_u,\xi_u)$, $u\in V(H)$. Changing one trial $\T_u$ can only affect the membership in $K$ of $u$ and of vertices sharing an edge with $u$. Hence only traces coming from original edges meeting this one-step neighborhood can change.  Since this neighborhood has size at most $1+r\Delta$, and every vertex is incident with at most $\Delta$ original edges, the number of affected original edges is at most $2r\Delta^2$ for $d$ large. Thus changing one trial changes $T_i$ by at most $\Lambda_i:=2(i+1)r\Delta^2$.
Moreover, if $T_i\geq s$, then $s$ trace incidences can be certified by revealing, for each of them, the trials in the corresponding original edge and in the one-step neighborhoods needed to verify the relevant activation and survival events.  This uses at most $\rho_i s$ trials, where $\rho_i:=2r\Delta$.

Put $M_i:=\E T_i=(1-p)\gamma n\mu_i$.  We have already shown that
$M_i\geq nd^{1/(3r)}$.  Let $a:=d^{-\eps/10}$.  Since
$\Delta\leq (r+1)d$ and $n\geq d^7$, the Talagrand correction term satisfies
\[
    20\Lambda_i\sqrt{\rho_iM_i}+64\Lambda_i^2\rho_i
    \leq \frac{1}{200}aM_i
\]
after increasing $\bar d=\bar d(r,\lambda)$ if necessary.  Indeed,
\[
    \frac{\Lambda_i\sqrt{\rho_iM_i}}{M_i}
    \leq C_r\frac{d^{5/2}}{M_i^{1/2}}
    \leq C_r d^{-1-1/(6r)}
    =o(a),
\]
and similarly $\Lambda_i^2\rho_i/M_i\leq C_r d^{-2-1/(3r)}=o(a)$, uniformly in
the allowed range of $\eps$.

Applying Talagrand's inequality with $t=aM_i/200$, we get
\[
\begin{aligned}
    \P\left(
        |T_i-\E T_i|>
        \frac{1}{100}d^{-\eps/10}\E T_i
    \right)
    &\leq
    4\exp\left(
        -\frac{t^2}{8\Lambda_i^2\rho_i(M_i+t)}
    \right)                                      \\
    &\leq
    4\exp\left(
        -C_r'\frac{a^2M_i}{\Delta^5}
    \right)
    \leq
    \exp\left(-\frac{\kappa_3 n}{d^6}\right).
\end{aligned}
\]
Here the last inequality follows from $a^2=d^{-\eps/5}$, $M_i\geq nd^{1/(3r)}$, $\Delta\leq (r+1)d$, and $\eps\leq 1/(10r)$.
Combining this estimate with the concentration of $|K|$, we obtain,
simultaneously for all $i\in[r]$,
\[
    d_{i}(H')
    =
    \left(1\pm \frac{1}{49}d^{-\eps/10}\right)
    \mu_i
\]
with probability at least $1-r\exp(-\kappa_3n/d^6)-\exp(-\kappa_2n/d^4)$.

In particular, for $i=r$, by increasing $\bar{d}$ if necessary we have
    \[
    \frac{(1-p)^r}{(1-p^r)^{r+1}}
    =
    1\pm \frac{1}{100}d^{-\eps/10}.
    \]
Thus
\[
    d_{r}(H')
    =
    \left(1\pm \frac{1}{30}d^{-\eps/10}\right)d\gamma^r.
\]

It remains to remove the vertices that are incident to too much edges of size less than $r+1$. For $1\leq i\leq r$, define
\[
    \sigma_i=
    \gamma^i d^{i/r}
    \sum_{k=i}^r c_k\binom{k}{i}\eps^{k-i}.
\]
Since $D_r\geq d_r(H)=d$, we have $c_r\geq1$. Since $\gamma\geq 1/2$, the term $k=r$ gives
\[
    \sigma_i\geq  \gamma^id^{i/r}\binom{r}{i} c_r\eps^{r-i}  \geq \dfrac{\eps^r}{2^r} d^{i/r}.
\]
Let
\[
    S_i
    =
    \left\{
        v\in K:
        d_{i,H'}(v)>\left(1+\frac{1}{20}d^{-\eps/10}\right)\sigma_i
    \right\}.
\]
We shall show that the union of the sets $S_i$ is very small.

Fix $v\in V(H)$ and $i\in[r]$. For $i\leq k\leq r$, an edge $e\in E_k(H)$ with $v\in e$ and any $U\in\binom{e\setminus\{v\}}{i}$, let $M^*(e,v,U)$ be the event
\[
    \bigcap_{u\in U}\Surv(u,e)
    \cap
    \bigcap_{w\in e\setminus(\{v\}\cup U)}\{X_w=1\}.
\]
This is the part of $M(e,v,U)$ not involving the event $\Surv(v,e)$. The reason for removing $\Surv(v,e)$ is that the events involving the vertex $v$ are shared by all edges incident with $v$, and hence are not independent as $e$ varies.  The event $M^*(e,v,U)$ only records the behavior of the vertices in $e\setminus\{v\}$: the vertices in $U$ survive all dangers except possibly $e$, while the remaining vertices of $e\setminus(\{v\}\cup U)$ are activated.  Thus, when $v\in K$, every actual $i$-trace of $H'$ containing $v$ is counted by one of the events $M^*(e,v,U)$. Define
\[
    P_{v,i}
    =
    \sum_{k=i}^r
    \sum_{\substack{e\in E_k(H)\\ v\in e}}
    \sum_{U\in\binom{e\setminus\{v\}}{i}}
    \indicator_{M^*(e,v,U)}.
\]
If $v\in K$, then $d_{i,H'}(v)\leq P_{v,i}$.
Set
\[
    Z_e=Z_e(v)=
    \sum_{U\in\binom{e\setminus\{v\}}{i}}
    \indicator_{M^*(e,v,U)}.
\]

\begin{claim}\label{claim:edge-level-independence}
For fixed $v$ and $i$, the variables $Z_e(v)$, indexed by edges $e\ni v$, are mutually independent. Moreover $Z_e(v)\in\{0,1\}$ for every $e$.
\end{claim}

\begin{claimproof}
For a fixed edge $e\ni v$, at most one of the events $M^*(e,v,U)$ can occur, since the event specifies exactly which $i$ vertices of $e\setminus\{v\}$ survive into $K$. Hence $Z_e(v)\in\{0,1\}$.

For an edge $e\ni v$, let $\T(e,v)$ be the set of trials used by the variables contributing to $Z_e(v)$. These consist of the trials of vertices in $e\setminus\{v\}$, together with the activation trials of vertices lying in edges $g\neq e$ which meet $e\setminus\{v\}$.

Suppose, for a contradiction, that distinct edges $e,f\ni v$ have overlapping determining trials after the trials involving $v$ are removed. First, if
\[
    (e\setminus\{v\})\cap(f\setminus\{v\})\neq\varnothing,
\]
then $e$ and $f$ share $v$ and another vertex, giving a forbidden $C_2$.
Second, suppose there is an auxiliary edge $g$ meeting both $e\setminus\{v\}$ and $f\setminus\{v\}$. Then $e,g,f$ form a $C_3$ through the two side intersections and $v$, unless two of these intersection vertices coincide, in which case we already have a $C_2$.
Finally, suppose there are auxiliary edges $g_e$ and $g_f$, where $g_e$ meets $e\setminus\{v\}$, $g_f$ meets $f\setminus\{v\}$, and $g_e\cap g_f\neq\varnothing$. Then $e,g_e,g_f,f$ form a $C_4$ through $v$, the two side intersections, and a vertex of $g_e\cap g_f$, unless a repeated intersection yields a  $C_2$ or $C_3$. All possibilities contradict uncrowdedness. Thus the determining trial sets are disjoint, and the variables $Z_e(v)$ are independent.
\end{claimproof}

By \Cref{claim:edge-level-independence}, $P_{v,i}$ is a sum of independent Bernoulli variables. Furthermore, since $p=\eps d^{-1/r}=o(d^{-\eps/10})$ uniformly in the allowed range of $\eps$, we have
\[
    \left(\frac{1-p}{1-p^k}\right)^i
    \leq
    1+\frac{1}{100}d^{-\eps/10}
\]
for every $1\leq i\leq k\leq r$, after increasing $\bar d$. Hence
\begin{align*}
    \E P_{v,i}
    &=
    \sum_{k=i}^r
    d_{k,H}(v)\binom{k}{i}
    \left(\frac{(1-p)\gamma}{1-p^k}\right)^i
    p^{k-i} \\
    &\leq
    \left(1+\frac{1}{100}d^{-\eps/10}\right)
    \gamma^i d^{i/r}
    \sum_{k=i}^r c_k\binom{k}{i}\eps^{k-i} \\
    &=
    \left(1+\frac{1}{100}d^{-\eps/10}\right)\sigma_i.
\end{align*}
By~\Cref{thm:chernoff},
\[
    \P\left(
        P_{v,i}>
        \left(1+\frac{1}{20}d^{-\eps/10}\right)\sigma_i
    \right)
    \leq
    \exp\left(-\kappa_4d^{-\eps/5}\sigma_i\right)
    \leq
    \exp\left(-\kappa_5d^{\frac{1}{2r}}\right).
\]
Thus
\[
    \E |S_i|\leq n\exp\left(-\kappa_5d^{\frac{1}{2r}}\right).
\]
By Markov's inequality, with probability at least
$1-r\exp(-d^{\frac{1}{4r}})$,
\[
    |S_i|
    \leq
    n\exp\left(-\kappa_5d^{\frac{1}{4r}}\right),\qquad \forall \; i\in [r].
\]
Put $S=\bigcup_{i=1}^r S_i$. Then by increasing $\bar{d}$ if necessary we have
$$|S|\leq n\exp\left(-\kappa_6d^{\frac{1}{4r}}\right)\leq \dfrac{1}{100}d^{-\eps/10} (1-p)\gamma n.$$

Finally define
\[
    V(H'')=K\setminus S,
    \qquad
    E(H'')=\{e\in E(H'):e\subseteq V(H'')\}.
\]
Then $H''$ is an induced subhypergraph of $H'$, hence is uncrowded by \Cref{lem:trace-uncrowded}.  Also by \Cref{claim:subhypergraphindependent}, every independent set $I'$ in $H''$ gives an independent set $I\cup I'$ in $H$.

Therefore
\[
    |V(H'')|= (1\pm d^{-\eps/10})(1-p)\gamma n.
\]

Let $T_r(F)=(r+1)|E_r(F)|$ for a bounded hypergraph $F$.
Each top-layer edge of $H'$ containing a vertex of $S$ is destroyed when
we pass to $H''=H'[K\setminus S]$.  A vertex of $H'$ is incident with at
most $\Delta$ trace edges, and hence deleting one vertex removes at most
$(r+1)\Delta\leq (r+1)^2d$ contributions to $T_r$.  Thus
\[
    T_r(H')-T_r(H'')\leq (r+1)^2d|S|.
\]
Since $|S|\leq n\exp(-\kappa_6d^{1/(4r)})$, by increasing $\bar d$ if necessary we may assume
\[
    (r+1)^2d|S|
    \le
    \frac{1}{1000}d^{-\eps/10}nd\gamma^{r+1}.
\]
Also
\[
    |S| \leq \frac{1}{1000}d^{-\eps/10}(1-p)\gamma n.
\]
Since
\[
    |K|=
    \left(1\pm \frac1{100}d^{-\eps/10}\right)(1-p)\gamma n,
\]
we have
\[
    \frac{|K|}{|K\setminus S|}
    =
    1+O(d^{-\eps/10}).
\]
More precisely, after increasing $\bar d$,
\[
    \frac{|K|}{|K\setminus S|}
    \le
    1+\frac1{200}d^{-\eps/10}.
\]
Moreover, using $|V(H'')|\geq \frac34\gamma n$,
\[
    \frac{T_r(H')-T_r(H'')}{|V(H'')|}
    \le
    \frac1{200}d^{-\eps/10}d\gamma^r.
\]
Since we already know
\[
    d_r(H')
    =
    \left(1\pm \frac1{30}d^{-\eps/10}\right)d\gamma^r,
\]
it follows that
\[
    d_r(H'')
    =
    \left(1\pm \frac1{20}d^{-\eps/10}\right)d\gamma^r.
\]

It remains only to verify the maximum degree condition.  By the definition
of the sets $S_i$, every vertex $v\in V(H'')$ satisfies
\[
    d_{i,H''}(v)
    \leq
    d_{i,H'}(v)
    \leq
    \left(1+\frac{1}{20}d^{-\eps/10}\right)\sigma_i.
\]
Therefore
\[
    D_{i}(H'')
    \leq
    \left(1+\frac{1}{20}d^{-\eps/10}\right)
    \gamma^i d^{i/r}
    \sum_{k=i}^r c_k\binom{k}{i}\eps^{k-i}.
\]
Using the lower bound
\[
    {d_{r}(H'')}^{i/r}
    \geq
    \left(1-\frac{1}{20}d^{-\eps/10}\right)
    \gamma^i d^{i/r},
\]
we get, for every $1\leq i\leq r$,
\[
    \frac{D_{i}(H'')}{{d_{r}(H'')}^{i/r}}
    \leq
    (1+d^{-\eps/10})
    \sum_{k=i}^r c_k\binom{k}{i}\eps^{k-i}.
\]
All failure probabilities above, summed over $i\in[r]$, are less than $1$ once $d\geq \bar{d}(r,\lambda)$ and $n\geq d^7$. Hence there is a realization for which all desired properties hold simultaneously. This gives the required pair $(I,H'')$.
\end{proof}

The coefficients given by \Cref{lem:nibble} enjoy good inductive upper bounds.

\begin{lemma}\label{lem:coefficientofmaxavg}
Suppose $P(r,0)\geq 0$ and $P(k,0)=0$ for every $k<r$.
If
\[
    P(k,i)\leq
    \sum_{\ell=k}^rP(\ell,i-1)\binom{\ell}{k}\lambda^{\ell-k}
    \qquad (1\leq k\leq r,\ i\geq 1),
\]
then
\[
    P(k,i)\leq
    P(r,0)\binom{r}{k}(i\lambda)^{r-k}
    \qquad (1\leq k\leq r,
    \ i\geq 0).
\]
\end{lemma}

\begin{proof}
The assertion is immediate for $i=0$.
Assume it is known at time $i-1$.
Then
\begin{align*}
    P(k,i)
    &\leq
    \sum_{\ell=k}^r
    P(r,0)\binom{r}{\ell}((i-1)\lambda)^{r-\ell}
    \binom{\ell}{k}\lambda^{\ell-k}  \\
    &=
    P(r,0)\binom{r}{k}\lambda^{r-k}
    \sum_{\ell=k}^r
    \binom{r-k}{\ell-k}(i-1)^{r-\ell} \\
    &=
    P(r,0)\binom{r}{k}\lambda^{r-k}i^{r-k}
    =
    P(r,0)\binom{r}{k}(i\lambda)^{r-k}.
\end{align*}
This proves the induction.
\end{proof}

The following lemma records the scale gain during the initial top-layer cleaning phase.

\begin{lemma}\label{lem:cleaning}
    For every $r\geq2$ there exists $d_{\mathrm{cl}}=d_{\mathrm{cl}}(r)$ such that the following holds for all $d\geq d_{\mathrm{cl}}$.  
    Let $H_0$ be an $n$-vertex $(r+1)$-uniform hypergraph with $d_r(H_0)=d$.  Suppose that, for $ 0\leq j<t $, $H_{j+1}$ is obtained by deleting a vertex $v_j$ such that \[d_{r,H_j}(v_j) > (1+\theta)d_r(H_j),\] where
    $0<\theta<1$.  
    Let
    \[
        n_i=|V(H_i)|,
        \qquad \dr{i}=d_r(H_i),
        \qquad R_i=\frac{n}{n_i},
        \qquad Q_i=\frac{n_i}{\dr{i}^{1/r}}.
    \]
    If $n_i\geq d^2$ for all $0\leq i\leq t$, then
    \[
        Q_i\geq R_i^{\theta/3}\frac{n}{d^{1/r}},
        \qquad 0\leq i\leq t.
    \]
\end{lemma}

\begin{proof}
    Since $H_i$ is $(r+1)$-uniform,
    \[
        (n_i-1)\dr{i+1}
        \leq n_i\dr{i}-(r+1)(1+\theta)\dr{i}.
    \]
    Thus
    \[
        \frac{\dr{i+1}}{\dr{i}}
        \leq
        1-\frac{r+(r+1)\theta}{n_i-1}.
    \]
    Since $n_i\geq d^2$ and $d$ is sufficiently large,
    \begin{align*}
        \frac{Q_{i+1}}{Q_i}
        &=
        \frac{n_i-1}{n_i}\left(\frac{\dr{i}}{\dr{i+1}}\right)^{1/r}  \\
        &\geq
        \frac{n_i-1}{n_i}
        \left(1-\frac{r+(r+1)\theta}{n_i-1}\right)^{-1/r}          \\
        &\geq
        \frac{n_i-1}{n_i}
        \left(1+\frac{r+(r+1)\theta}{r(n_i-1)}\right)              \\
        &=
        1+\frac{(r+1)\theta}{rn_i}
        \geq
        \left(\frac{n_i}{n_i-1}\right)^{\theta/3}.
    \end{align*}
    Iterating over the cleaning steps gives
    \[
        Q_i
        \geq
        \frac{n}{d^{1/r}}
        \prod_{j=0}^{i-1}\left(\frac{n_j}{n_j-1}\right)^{\theta/3}
        =
        R_i^{\theta/3}\frac{n}{d^{1/r}},
    \]
    for every $0\leq i\leq t$.
\end{proof}

\section{Proof of \Cref{thm:main}}\label{sec:mainproof}

We now prove \Cref{thm:main}.  Fix $r$ and $\eta>0$.  Choose constants $\eps_0$ and $\rho$ so that
\[
    0<\rho\ll_r \eps_0\ll_r \eta, \qquad \eps_0<\frac{1}{100r}, \qquad \rho<\frac{1}{1000r^2},
\]
and, more explicitly,
\[
    11\eps_0<\frac{\eta}{2}, \qquad 20r\rho<\frac{\eta}{2}.
\]
Set $\lambda_*:=\rho\left(\eps_0/5\right)^{(r-1)/r}$. 

Let $\bar d(r,\lambda_*)$ be the threshold in \Cref{lem:nibble}.
Choose $d_*$ sufficiently large so that
\[
    d_*^{\eps_0/5}\geq \bar d(r,\lambda_*),
\]
and so that all estimates below hold for every $d\geq d_*$.

Let $G$ be an $n$-vertex uncrowded $(r+1)$-uniform hypergraph with average degree $d=d_r(G)\geq d_*$.  We may assume \[n\geq d^{30}.\] Indeed, otherwise we apply the argument below to a disjoint union of sufficiently many copies of $G$. The average degree and uncrowdedness are preserved, and one of the copies contains an independent set of at least the average size obtained in the union.

Throughout the proof put
\[  L:=\log d^{1/r},
    \qquad
    h:=\rho L^{-(r-1)/r},
    \qquad
    T:=\left\lfloor(1-10\eps_0)\frac{L^{1/r}}{h}\right\rfloor.\]
By increasing $d_*$ if necessary, we may assume that for every $d\geq d_*$ and every sequence $a_j\geq d^{\eps_0/5}$, $0\leq j<T$, the following elementary estimates hold for every $0\leq k\leq T$:
\begin{equation}\label{eq:B1}
    T\leq 2\rho^{-1}L,\quad
    1+\sum_{j<k}a_j^{-h/10}\leq \prod_{j<k}\left(1+a_j^{-h/10}\right) \leq 1+\frac{\eps_0}{100},\quad
    \sum_{j<k}\frac{h}{a_j^{1/r}}\leq\frac{\eps_0}{100}.
\end{equation}

\subsection*{Initial top-layer cleaning}
We first perform a deterministic cleaning of the top layer.

\begin{breakablealgorithm}
\caption{Top-layer cleaning procedure}\label{algorithm:clean}
\begin{flushleft}
\textbf{Input}: An $n$-vertex $(r+1)$-uniform uncrowded hypergraph $G$ with average degree $d=d_r(G)$, and parameter $\eps_0$.\\
\textbf{Output}: A residual hypergraph $H$.
\end{flushleft}

\begin{enumerate}[itemsep=.2cm]
    \item Initialize
    \[
        H_0=G,
        \qquad R_0=1,
        \qquad i=0.
    \]
    \item \label{algstep:stop} If one of the following three stopping criteria holds, return $H_i$:
    \[
\begin{array}{ll}
    \hypertarget{stop:S1}{\textnormal{[S1]}}\,
    d_r(H_i)<d^{1-\eps_0/4}; \qquad
    \hypertarget{stop:S2}{\textnormal{[S2]}}\,
    R_i\geq d^{\eps_0/(10r)}; \qquad
    \hypertarget{stop:S3}{\textnormal{[S3]}}\,
    D_r(H_i)\leq (1+\eps_0/2)d_r(H_i).
\end{array}
    \]
    \item If no stopping criterion applies, then
        \[
            D_r(H_i)>(1+\eps_0/2)d_r(H_i),
        \]
        choose a vertex $v_i\in \text{argmax}_{v\in V(H_i)} d_{r,H_i}(v)$ and set
        \[
            H_{i+1}=H_i-v_i,
            \qquad
            R_{i+1}=R_i\frac{|V(H_i)|}{|V(H_i)|-1}.
        \]
    \item Update $i\leftarrow i+1$ and return to Step~\ref{algstep:stop}.
\end{enumerate}
\end{breakablealgorithm}
\vspace{0.3cm}

During this cleaning procedure the current hypergraph remains uniform.  We shall use the following immediate consequence of \Cref{lem:randomind}: There is a constant $C_r$ such that every $(r+1)$-uniform hypergraph $F$ with average degree $d_F$ satisfies
\begin{equation}\label{eq:uniform-alteration-section3}
    \alpha(F)\geq C_r^{-1}\frac{|V(F)|}{\max\{d_F,1\}^{1/r}}.
\end{equation}
Let $s$ be the stopping time of \Cref{algorithm:clean}.  At any cleaning stopping time, we have
\begin{equation}\label{eq:R-stopping-bound}
    R_s\leq 2d^{\eps_0/(10r)}.
\end{equation}
Indeed, if $s=0$, this is clear.  If $s>0$, then no stopping criterion held at time $s-1$, and in particular
\[
    R_{s-1}<d^{\eps_0/(10r)}.
\]
Since the last cleaning step deletes only one vertex,
\[
R_s=R_{s-1}\frac{|V(H_{s-1})|}{|V(H_{s-1})|-1}\leq 2R_{s-1}<2d^{\eps_0/(10r)}.
\]
We now consider the possible reasons for stopping. 
If $H_s$ satisfies \hyperlink{stop:S1}{\textnormal{[S1]}}, then
\[\max\{d_r(H_s),1\}\leq d^{1-\eps_0/4}.\]
By \eqref{eq:uniform-alteration-section3} and \eqref{eq:R-stopping-bound},
\[
    \alpha(G)\geq\alpha(H_s)\geq C_r^{-1}\frac{n}{R_s\max\{d_r(H_s),1\}^{1/r}}\geq(2C_r)^{-1}d^{3\eps_0/(20r)}\frac{n}{d^{1/r}}.
\]
For $d_*$ sufficiently large, the last expression is at least
\[
    r^{-1/r}\left(\frac{\log d}{d}\right)^{1/r}n,
\]
and hence the desired conclusion follows.

Next suppose that $H_s$ satisfies \hyperlink{stop:S2}{\textnormal{[S2]}}, and that \hyperlink{stop:S1}{\textnormal{[S1]}} does not hold.  Then
\[
    d_r(H_s)\geq d^{1-\eps_0/4}>1.
\]
Moreover, by \eqref{eq:R-stopping-bound} and the assumption $n\geq d^{30}$,
for every $0\leq j\leq s$,
\[
    |V(H_j)|\geq |V(H_s)|=\frac{n}{R_s}\geq\frac12 d^{30-\eps_0/(10r)}\geq d^2,
\]
provided $d_*$ is large. Hence \Cref{lem:cleaning} is applicable to the cleaning sequence with $\theta=\eps_0/2$.  Since $R_s\geq d^{\eps_0/(10r)}$, it gives
\[
    Q_s:=\frac{|V(H_s)|}{d_r(H_s)^{1/r}}\geq R_s^{\eps_0/6}\frac{n}{d^{1/r}}\geq d^{\eps_0^2/(60r)}\frac{n}{d^{1/r}}.
\]
Using \eqref{eq:uniform-alteration-section3},
\[
    \alpha(G)\geq \alpha(H_s)\geq C_r^{-1}Q_s \geq C_r^{-1}d^{\eps_0^2/(60r)}\frac{n}{d^{1/r}}.
\]
Again, for $d_*$ sufficiently large, this is at least
\[
    r^{-1/r}\left(\frac{\log d}{d}\right)^{1/r}n.
\]
Thus the theorem is proved in this case as well. 

We may therefore assume that \Cref{algorithm:clean} stops because \hyperlink{stop:S3}{\textnormal{[S3]}} holds, while neither \hyperlink{stop:S1}{\textnormal{[S1]}} nor \hyperlink{stop:S2}{\textnormal{[S2]}} holds.  
Let
\[ H_{\mathrm{cl}}:=H_s,\qquad n_0:=|V(H_{\mathrm{cl}})|,\qquad d_{(0)}:=d_r(H_{\mathrm{cl}}).
\]
Then
\begin{equation}\label{eq:cleaned-initial-properties}
    d_{(0)}\geq d^{1-\eps_0/4},\qquad D_r(H_{\mathrm{cl}}) \leq \left(1+\frac{\eps_0}{2}\right)d_{(0)},
\end{equation}
and \Cref{lem:cleaning} gives
\begin{equation}\label{eq:initial-Q-cleaned}
    \frac{n_0}{d_{(0)}^{1/r}}\geq \frac{n}{d^{1/r}}.
\end{equation}
The role of the cleaning procedure is precisely to reach this balanced top-layer regime while preserving the scale $|V(H)|/d_r(H)^{1/r}$.  From this point on, no further top-layer cleaning is performed.

\subsection*{The trace nibble phase}
We now run the trace nibble procedure.
\begin{breakablealgorithm}\label{algo:nibble-after-clean}
\caption{Multi-round trace nibble}
\begin{flushleft}
\textbf{Input}: The cleaned residual hypergraph $H_{\mathrm{cl}}$, the original average degree $d$, and the parameters $\eps_0,\rho$.

\textbf{Output}: An independent set $I\subseteq V(H_{\mathrm{cl}})$.
\end{flushleft}
\begin{enumerate}[label=(\arabic*)]
    \item Initialize
    \[
        H'_0=H_{\mathrm{cl}},\qquad I=\varnothing, \qquad k=0.
    \]

    \item If $k=T$, stop and output $I$.

    \item Apply \Cref{lem:nibble} to $H'_k$ with parameters
    \[
        \lambda=\lambda_*,\qquad \eps=h.
    \]
    Let $(I_k,H'_{k+1})$ be the pair obtained from \Cref{lem:nibble}.
    Update
    \[
        I\leftarrow I\cup I_k,\qquad k\leftarrow k+1,
    \]
    and return to Step~2.
\end{enumerate}
\end{breakablealgorithm}

The compatibility conclusion in \Cref{lem:nibble} ensures inductively that the final set $I$ produced by \Cref{algo:nibble-after-clean} is independent in $H_{\mathrm{cl}}$, and hence also independent in $G$.  It remains to verify that \Cref{lem:nibble} is applicable for all $T$ rounds and to estimate the size of $I$.

For $0\leq k\leq T$, provided the first $k$ nibble rounds have been performed, write
\[
    n_k=|V(H'_k)|,
    \qquad
    \dr{k}=d_r(H'_k),
    \qquad
    Q_k=\frac{n_k}{\dr{k}^{1/r}}.
\]
For $1\leq \ell\leq r$, define the normalized layer-degree coefficients
\[
    c_{\ell,k}:=\frac{D_\ell(H'_k)}{\dr{k}^{\ell/r}},
    \qquad 1\leq \ell\leq r.
\]
At time $k=0$, the hypergraph $H'_0=H_{\mathrm{cl}}$ is $(r+1)$-uniform.  Hence
\[
    c_{\ell,0}=0 \quad (\ell<r),
    \qquad
    c_{r,0}\leq 1+\frac{\eps_0}{2},
\]
and \eqref{eq:cleaned-initial-properties}--\eqref{eq:initial-Q-cleaned} give
\[
    d_{(0)}\geq d^{1-\eps_0/4},
    \qquad
    Q_0\geq \frac{n}{d^{1/r}}.
\]
For each round $j$, set
\[
    p_j:=\frac{h}{d_{(j)}^{1/r}},
    \qquad
    \gamma_j:=
    \prod_{\ell=1}^r(1-p_j^\ell)^{D_\ell(H'_j)},
    \qquad
    \delta_j:=
    \sum_{\ell=1}^r c_{\ell,j}h^\ell.
\]
Also put
\[
    A_k:=\prod_{j=0}^{k-1}\left(1+d_{(j)}^{-h/10}\right),
    \qquad
    A_0:=1.
\]
\begin{claim}\label{claim:section3-induction}
The trace nibble phase can be performed for all $T$ rounds.  Moreover, for
every $0\leq k\leq T$,
\[
    d^{\eps_0/5}\leq d_{(k)}\leq d,
    \qquad
    A_k\leq 1+\frac{\eps_0}{100},
\]
and for every $1\leq \ell\leq r$,
\begin{equation}\label{eq:coefficient-induction-section3}
    c_{\ell,k} \leq A_k\left(1+\frac{\eps_0}{2}\right) \binom r\ell (kh)^{r-\ell}. \quad (\text{Here }0^0=1. )
\end{equation}
\end{claim}
\begin{proof}
    We argue by induction on $k$.  The case $k=0$ follows from the properties of $H'_0$ recorded above.

    Assume that $0\leq k<T$, that the first $k$ nibble rounds have been performed, and that the estimates in the claim hold at all earlier times up to $k$.  We first check the hypotheses of \Cref{lem:nibble} for $H'_k$.

    The degree lower bound gives
    \[
    d_{(k)}\geq d^{\eps_0/5}\geq \bar d(r,\lambda_*).
    \]
    Moreover,
    \[
    \log d_{(k)}^{1/r}\geq\frac{\eps_0}{5}\log d^{1/r}=\frac{\eps_0}{5}L,
    \]
    and therefore
    \[
    h=\rho L^{-(r-1)/r}\geq\rho\left(\frac{\eps_0}{5}\right)^{(r-1)/r}\left(\log d_{(k)}^{1/r}\right)^{-(r-1)/r}=\lambda_* \left(\log d_{(k)}^{1/r}\right)^{-(r-1)/r}.
    \]
    Also $h\leq 1/(10r)$ for $d_*$ sufficiently large. 
    Next we verify the smallness of $\delta_k$.  By \eqref{eq:coefficient-induction-section3},
    \[
    \begin{aligned}
    \delta_k &\leq A_k\left(1+\frac{\eps_0}{2}\right) \sum_{\ell=1}^r \binom r\ell (kh)^{r-\ell}h^\ell        \\
    &= A_k\left(1+\frac{\eps_0}{2}\right) \left(((k+1)h)^r-(kh)^r\right).
    \end{aligned}
    \]
    Since $k<T$,
    \[
    (k+1)h\leq Th\leq (1-10\eps_0)L^{1/r}.
    \]
    Using $x^r-y^r\leq rx^{r-1}(x-y)$ for $x\geq y\geq0$, we get
    \[
    ((k+1)h)^r-(kh)^r \leq r\left((1-10\eps_0)L^{1/r}\right)^{r-1}h =r(1-10\eps_0)^{r-1}\rho \leq r\rho.
    \]
    Together with $A_k\leq 1+\eps_0/100$, this yields
    \begin{equation}\label{eq:delta-k-small-section3}
    \delta_k \leq 2r\rho \leq \frac1{100r}.
    \end{equation}

    The top-layer coefficient satisfies
    \[
    c_{r,k} \leq A_k\left(1+\frac{\eps_0}{2}\right) \leq \left(1+\frac{\eps_0}{100}\right) \left(1+\frac{\eps_0}{2}\right) < 1+\frac1{4r}.
    \]
    For $\ell<r$, \eqref{eq:coefficient-induction-section3} and $kh\leq L^{1/r}$ give
    \[
    c_{\ell,k} \leq 2\binom r\ell L^{(r-\ell)/r} < d_{(k)}^{1/r},
    \]
    provided $d_*$ is sufficiently large, since $d_{(k)}^{1/r}\geq d^{\eps_0/(5r)}$.

    It remains to check the vertex condition $n_k\geq d_{(k)}^7$.  Since the cleaning procedure did not stop by \hyperlink{stop:S2}{\textnormal{[S2]}}, we have
    \[
    n_0>|V(G)|d^{-\eps_0/(10r)} = nd^{-\eps_0/(10r)}.
    \]
    For $k=0$, this already gives $n_0\geq d_{(0)}^7$ for $d_*$ large, because $n\geq d^{30}$ and $d_{(0)}\leq d$.  For $k>0$, the vertex estimate in \Cref{lem:nibble} from the previous rounds gives
    \[
    n_k \geq n_0 \prod_{j=0}^{k-1} \left(1-d_{(j)}^{-h/10}\right) (1-p_j)\gamma_j.
    \]
    By the lower bound on $\gamma_j$ in \eqref{eq:lowerboundforgamma},
    \[
    -\log\gamma_j \leq \frac{\delta_j}{1-p_j} \leq \left(1+\frac{\eps_0}{50}\right)\delta_j
    \]
    for $d_*$ sufficiently large.  Summing the estimate for $\delta_j$ above over $j<k$ gives
    \[
    \sum_{j<k}\delta_j \leq 2(kh)^r \leq 2L \leq 2rL.
    \]
    Using \eqref{eq:B1}, we obtain
    \[
    n_k \geq n_0 \left(1-\sum_{j<k}d_{(j)}^{-h/10}\right) \left(1-\sum_{j<k}p_j\right) \exp(-3rL).
    \]
    Since $rL=\log d$, $n\geq d^{30}$, and $n_0\geq nd^{-\eps_0/(10r)}$, it follows, by increasing $d_*$ if necessary, that
    \[
    n_k\geq d^{20}.
    \]
    As $d_{(k)}\leq d$, this implies $n_k\geq d_{(k)}^7$.

    Thus all hypotheses of \Cref{lem:nibble} hold for $H'_k$.  Applying \Cref{lem:nibble} gives the next pair $(I_k,H'_{k+1})$.

    We now verify the inductive estimates at time $k+1$.  First, \eqref{eq:B1} gives
    \[
    A_{k+1} = \prod_{j=0}^{k} \left(1+d_{(j)}^{-h/10}\right)
    \leq 1+\frac{\eps_0}{100}.
    \]
    The maximum-degree conclusion of \Cref{lem:nibble} gives, for every $1\leq \ell\leq r$,
    \[
    c_{\ell,k+1} \leq \left(1+d_{(k)}^{-h/10}\right) \sum_{m=\ell}^r c_{m,k}\binom m\ell h^{m-\ell}.
    \]
    Dividing by the accumulated factor $A_{k+1}$, and applying \Cref{lem:coefficientofmaxavg} with initial value $c_{r,0}\leq 1+\eps_0/2$, yields
    \[
    c_{\ell,k+1} \leq A_{k+1}\left(1+\frac{\eps_0}{2}\right) \binom r\ell ((k+1)h)^{r-\ell}.
    \]
    It remains to control $d_{(k+1)}$.  Since $D_r(H'_k)\geq d_{(k)}$, we have
    \[
    \gamma_k^r \leq (1-p_k^r)^{r d_{(k)}} \leq e^{-rh^r}.
    \]
    By the top-layer average-degree estimate in \Cref{lem:nibble},
    \[
    d_{(k+1)} \leq \left(1+\frac1{20}d_{(k)}^{-h/10}\right) d_{(k)}\gamma_k^r.
    \]
    Since $d_{(k)}\geq d^{\eps_0/5}$, the term $d_{(k)}^{-h/10}$ is stretched-exponentially small in $L^{1/r}$, while $h^r=\rho^rL^{-(r-1)}$.  Hence, by increasing $d_*$ if necessary,
    \[
    \frac1{20}d_{(k)}^{-h/10} \leq \frac12 rh^r.
    \]
    Therefore
    \[
    d_{(k+1)} \leq d_{(k)} \exp\left(\frac12 rh^r-rh^r\right) \leq d_{(k)} \leq d.
    \]
    For the lower bound, the lower estimate in \Cref{lem:nibble} gives
    \[
    d_{(k+1)} \geq d_{(0)} \prod_{j=0}^{k} \left(1-\frac1{20}d_{(j)}^{-h/10}\right) \prod_{j=0}^{k}\gamma_j^r .
    \]
    Taking logarithms and using \eqref{eq:lowerboundforgamma}, we obtain
    \[
    \begin{aligned}
    \log d_{(k+1)} &\geq \log d_{(0)} + \sum_{j=0}^{k} \log\left(1-\frac1{20}d_{(j)}^{-h/10}\right) - r\sum_{j=0}^{k}\frac{\delta_j}{1-p_j}.
    \end{aligned}
    \]
    The first term satisfies
    \[
    \log d_{(0)} \geq \left(1-\frac{\eps_0}{4}\right)\log d.
    \]
    By \eqref{eq:B1}, the logarithmic contribution of the product error is at least $-\frac{\eps_0}{100}\log d$ for $d_*$ sufficiently large.  Also,
    \[
    \sum_{j=0}^{k}\delta_j \leq \left(1+\frac{3\eps_0}{5}\right)((k+1)h)^r \leq \left(1+\frac{3\eps_0}{5}\right)(1-10\eps_0)^rL,
    \]
    and $1/(1-p_j)\leq 1+\eps_0/50$. Since $rL=\log d$, we get
    \[
    \log d_{(k+1)} \geq \left(1-\frac{\eps_0}{4} -\frac{\eps_0}{100} - \left(1+\frac{4\eps_0}{5}\right)(1-10\eps_0)^r \right)\log d.
    \]
    By choosing $\eps_0$ sufficiently small in terms of $r$, the coefficient in brackets is larger than $\eps_0/5$. Hence
    \[
    d_{(k+1)}\geq d^{\eps_0/5}.
    \]
    This closes the induction and proves the claim.
\end{proof}

\subsection*{Summing the nibble rounds}

We now estimate the size of the independent set produced by the $T$ nibble
rounds.  From the vertex and top-layer degree estimates in \Cref{lem:nibble},
\[
    Q_{k+1} \geq (1-2\dr{k}^{-h/10})(1-p_k)Q_k, \qquad p_k=\frac{h}{\dr{k}^{1/r}}.
\]
Using \eqref{eq:B1}, we have
\[
    Q_k\geq (1-\eps_0/10)Q_0\geq (1-\eps_0/10)\frac{n}{d^{1/r}}
\]
uniformly for $0\leq k\leq T$.  
The independent set obtained in round $k$ has size at least
\[
    |I_k| \geq (1-3h-3\delta_k)p_kn_k = (1-3h-3\delta_k)hQ_k.
\]
By \eqref{eq:delta-k-small-section3} and $h\leq \rho$, we have
\[
    1-3h-3\delta_k\geq 1-10r\rho.
\]
Therefore the final independent set $I$ satisfies
\begin{align*}
    |I|&=\sum_{k=0}^{T-1}|I_k| \\
    &\geq (1-10r\rho)(1-\eps_0/10)\sum_{k=0}^{T-1}h\frac{n}{d^{1/r}}  \\
    &\geq (1-10r\rho-\eps_0)Th \frac{n}{d^{1/r}}.
\end{align*}
Since
\[
    T h \geq (1-10\eps_0)L^{1/r}-h,
\]
and $h=o(L^{1/r})$, increasing $d_*$ if necessary gives
\[
    Th\geq (1-10\eps_0-\rho)L^{1/r}.
\]
Consequently,
\[
    |I| \geq (1-20r\rho-11\eps_0) L^{1/r}\frac{n}{d^{1/r}}.
\]
Notice that $
    L^{1/r} = \left(\log d^{1/r}\right)^{1/r} = r^{-1/r}(\log d)^{1/r}$.
By our choice of $\rho$ and $\eps_0$,
\[
    20r\rho+11\eps_0<\eta.
\]
Hence 
\[
    |I| \geq (1-\eta)r^{-1/r}\left(\frac{\log d}{d}\right)^{1/r}n.
\]
This proves \Cref{thm:main}. 

\section{Open problems and concluding remarks}
Kelly and Postle suggested a conjecture \cite[Conjecture 2.2]{KP24} of local Shearer bound, which was proved by Martinsson and Steiner \cite{MS}, a strengthened Shearer's bound on the independence number for triangle-free graphs in the following way:
\begin{theo}[Martinsson and Steiner {\cite[Theorem 1.2]{MS}}]\label{thm:local_Shearer_bound}
    For every triangle-free graph $G$ there exists a probability distribution $\D$ on the independent sets of $G$ such that $$\P_{I\sim \D}[v\in I]\geq (1-o(1))\frac{\log d_G(v)}{d_G(v)}$$ for every $v\in V(G)$. Here the $o(1)$ represents a function of $d_G(v)$ that tends to $0$ as the degree grows.
\end{theo}
In fact, they proved a more general theorem as follows: 
\begin{theo}[Martinsson and Steiner {\cite[Theorem 2.1]{MS}}]\label{thm:weighted_local_Shearer_bound}
    For every triangle-free graph $G$ and every strictly positive weight function $w:V(G)\to \R_+$ on the vertices there exists a probability distribution $\D$ on the independent sets of $G$ such that $$\P_{I\sim \D}[v\in I]\geq f\left(\frac{w(N_G(v))}{w(v)}\right)$$ for every vertex $v\in V(G)$, where $N_G(v)$ denotes the neighborhood of $v$,
$w(S):=\sum_{u\in S}w(u)$ for $S\subseteq V(G)$, and $f(x)=(1-o_x(1))\frac{\log x}{x}$ as $x\to\infty$.
\end{theo}
Naturally, we conjecture the following local version of \Cref{thm:main}: 
\begin{conj}\label{conj}
   For every $r\geq2$, there exists a function $f_r:[0,\infty)\to[0,1]$ satisfying
\[
    f_r(x)
    =
    (1-o_x(1))r^{-1/r}
    \left(\frac{\log x}{x}\right)^{1/r}
    \qquad\text{as }x\to\infty,
\]
such that every uncrowded $(r+1)$-uniform hypergraph $H$ admits a probability distribution $\mathcal D$ on its independent sets for which
\[
    \P_{I\sim\mathcal D}(v\in I)
    \geq
    f_r\bigl(d_{r,H}(v)\bigr)
\]
for every $v\in V(H)$.
\end{conj}
\printbibliography

@article {Ajtai1980Ramsy,
    AUTHOR = {Ajtai, M. and Koml\'os, J. and Szemer\'edi, E.},
     TITLE = {A note on {R}amsey numbers},
   JOURNAL = {J. Combin. Theory Ser. A},
  FJOURNAL = {Journal of Combinatorial Theory. Series A},
    VOLUME = {29},
      YEAR = {1980},
    NUMBER = {3},
     PAGES = {354--360},
   MRCLASS = {05C55 (05C35)},
  MRNUMBER = {600598},
MRREVIEWER = {J.\ E.\ Graver},
       %URL = {https://doi.org/10.1016/0097-3165(80)90030-8},
}

@article {Shearer1983ind1,
    AUTHOR = {Shearer, J. B.},
     TITLE = {A note on the independence number of triangle-free graphs},
   JOURNAL = {Discrete Math.},
  FJOURNAL = {Discrete Mathematics},
    VOLUME = {46},
      YEAR = {1983},
    NUMBER = {1},
     PAGES = {83--87},
   MRCLASS = {05C99},
  MRNUMBER = {708165},
MRREVIEWER = {Linda\ Lesniak},
      % URL = {https://doi.org/10.1016/0012-365X(83)90273-X},
}

@article {Spencer1982hyper,
    AUTHOR = {Ajtai, M. and Koml\'os, J. and Pintz, J. and Spencer, J. and
              Szemer\'edi, E.},
     TITLE = {Extremal uncrowded hypergraphs},
   JOURNAL = {J. Combin. Theory Ser. A},
  FJOURNAL = {Journal of Combinatorial Theory. Series A},
    VOLUME = {32},
      YEAR = {1982},
    NUMBER = {3},
     PAGES = {321--335},
   MRCLASS = {05C65},
  MRNUMBER = {657047},
MRREVIEWER = {F.\ Sterboul},
}

@inproceedings {Rodl1995hyper,
    AUTHOR = {Duke, R. A. and Lefmann, H. and R\"odl, V.},
     TITLE = {On uncrowded hypergraphs},
 BOOKTITLE = {Proceedings of the {S}ixth {I}nternational {S}eminar on
              {R}andom {G}raphs and {P}robabilistic {M}ethods in
              {C}ombinatorics and {C}omputer {S}cience, ``{R}andom {G}raphs
              '93'' ({P}ozna\'n, 1993)},
   JOURNAL = {Random Structures Algorithms},
  FJOURNAL = {Random Structures \& Algorithms},
    VOLUME = {6},
      YEAR = {1995},
    NUMBER = {2-3},
     PAGES = {209--212},
   MRCLASS = {05C65 (05C80)},
  MRNUMBER = {1370956},
       %URL = {https://doi.org/10.1002/rsa.3240060208},
}

@article {Verstraete2026hyper,
    AUTHOR = {Verstraete, J. and Wilson, C.},
     TITLE = {Independent sets in hypergraphs},
   JOURNAL = {Random Structures Algorithms},
  FJOURNAL = {Random Structures \& Algorithms},
    VOLUME = {68},
      YEAR = {2026},
    NUMBER = {1},
     PAGES = {Paper No. e70047, 9},
   MRCLASS = {05C69 (05C65)},
  MRNUMBER = {5022514},
       %URL = {https://doi.org/10.1002/rsa.70047},
}

@unpublished{DhawanJanzerMethuku2025,
	author = {A. Dhawan and O. Janzer and A. Methuku},
	title = {Independent sets and colorings of $K_{t,t,t}$-free graphs},
	howpublished = {\url{https://arxiv.org/abs/2511.17191} (preprint)},
	date = {2025},
}

@unpublished{Dhawan2026independence,
	author = {A. Dhawan and A. Methuku and M.-Q. Vo},
	title = {The independence number of uncrowded hypergraphs: bounds matching the shattering threshold},
	howpublished = {\url{https://arxiv.org/abs/2606.18048} (preprint)},
	date = {2026},
}

@unpublished{YuZhang2026Transfer, 
	author = {J. Yu and J. Zhang},
	title = {Hypergraph independence bounds: from maximum degree to average degree},
	howpublished = {\url{https://arxiv.org/abs/2604.28046} (preprint)},
	date = {2026},
}

@article{NV21,
    AUTHOR = {Nie, J. and Verstraete, J.},
     TITLE = {Randomized greedy algorithm for independent sets in regular uniform hypergraphs with large girth},
   JOURNAL = {Random Structures Algorithms},
  FJOURNAL = {Random Structures \& Algorithms},
    VOLUME = {59},
      YEAR = {2021},
    NUMBER = {1},
     PAGES = {79--95},
   MRCLASS = {05C65 (05C69 68W20)},
  MRNUMBER = {4270996},
MRREVIEWER = {W.\ G.\ Brown},
       %URL = {https://doi.org/10.1002/rsa.20994},
}

@article {CT91,
    AUTHOR = {Caro, Y. and Tuza, Z.},
     TITLE = {Improved lower bounds on {$k$}-independence},
   JOURNAL = {J. Graph Theory},
  FJOURNAL = {Journal of Graph Theory},
    VOLUME = {15},
      YEAR = {1991},
    NUMBER = {1},
     PAGES = {99--107},
   MRCLASS = {05C35 (05C65)},
  MRNUMBER = {1090733},
MRREVIEWER = {Michael\ Jacobson},
      % URL = {https://doi.org/10.1002/jgt.3190150110},
}

@article {MR14,
    AUTHOR = {Molloy, M. and Reed, B.},
     TITLE = {Colouring graphs when the number of colours is almost the
              maximum degree},
   JOURNAL = {J. Combin. Theory Ser. B},
  FJOURNAL = {Journal of Combinatorial Theory. Series B},
    VOLUME = {109},
      YEAR = {2014},
     PAGES = {134--195},
   MRCLASS = {05C15},
  MRNUMBER = {3269906},
MRREVIEWER = {Daqing\ Yang},
      % URL = {https://doi.org/10.1016/j.jctb.2014.06.004},
}

@article{GamarnikGoldberg2010Greedy,
    AUTHOR = {Gamarnik, D. and Goldberg, D. A.},
     TITLE = {Randomized greedy algorithms for independent sets and
              matchings in regular graphs: exact results and finite girth
              corrections},
   JOURNAL = {Combin. Probab. Comput.},
  FJOURNAL = {Combinatorics, Probability and Computing},
    VOLUME = {19},
      YEAR = {2010},
    NUMBER = {1},
     PAGES = {61--85},
   MRCLASS = {05C85 (05C69 05C70 68W20)},
  MRNUMBER = {2575098},
MRREVIEWER = {T.\ R. S. Walsh},
}

@article{LauerWormald2007LargeGirth,
    AUTHOR = {Lauer, J. and Wormald, N.},
     TITLE = {Large independent sets in regular graphs of large girth},
   JOURNAL = {J. Combin. Theory Ser. B},
  FJOURNAL = {Journal of Combinatorial Theory. Series B},
    VOLUME = {97},
      YEAR = {2007},
    NUMBER = {6},
     PAGES = {999--1009},
   MRCLASS = {05C69 (05C80)},
  MRNUMBER = {2354714},
MRREVIEWER = {Hamed\ Hatami},      
}

@article{KostochkaMubayiVerstraete2014Independent,
    AUTHOR = {Kostochka, A. and Mubayi, D. and Verstraete,
              J.},
     TITLE = {On independent sets in hypergraphs},
   JOURNAL = {Random Structures Algorithms},
  FJOURNAL = {Random Structures \& Algorithms},
    VOLUME = {44},
      YEAR = {2014},
    NUMBER = {2},
     PAGES = {224--239},
   MRCLASS = {05C65 (05C55 05C69 05D10 51E10)},
  MRNUMBER = {3158630},
MRREVIEWER = {Theodore\ C.\ Enns},
}

@article {MS,
    AUTHOR = {Martinsson, A. and Steiner, R.},
     TITLE = {Random independent sets in triangle-free graphs},
   JOURNAL = {Forum Math. Sigma},
  FJOURNAL = {Forum of Mathematics. Sigma},
    VOLUME = {13},
      YEAR = {2025},
     PAGES = {Paper No. e156, 19},
   MRCLASS = {05C69 (05C07 05C15 05C72)},
  MRNUMBER = {4963951},
}

@article {KP24,
    AUTHOR = {Kelly, T. and Postle, L.},
     TITLE = {Fractional coloring with local demands and applications to
              degree-sequence bounds on the independence number},
   JOURNAL = {J. Combin. Theory Ser. B},
  FJOURNAL = {Journal of Combinatorial Theory. Series B},
    VOLUME = {169},
      YEAR = {2024},
     PAGES = {298--337},
   MRCLASS = {05C15 (05C69)},
  MRNUMBER = {4776369},
MRREVIEWER = {James\ Tuite},
}

\end{document}